\documentclass{amsart}

\usepackage[english]{babel}
\usepackage[colorlinks]{hyperref}
\hypersetup{
    linkcolor=red, 
    citecolor=blue, 
}

\usepackage{multicol}
\usepackage{graphicx}

\usepackage{amsmath}

\usepackage[latin1]{inputenc}
\usepackage{amsthm}
\usepackage{amsfonts}
\usepackage{enumitem}
\usepackage{xcolor}

\newtheorem*{theorem*}{Theorem}
\newtheorem{theorem} {Theorem}[section]
\newtheorem{corollary}[theorem]{Corollary}
\newtheorem{lemma}[theorem]{Lemma}
\newtheorem{proposition}[theorem]{Proposition}

\newtheorem{remark}[theorem]{Remark}

\newtheorem{definition}[theorem]{Definition}

\newcommand{\conv}  {\operatorname{conv} }
\newcommand{\vol}  {\operatorname{vol} }
\newcommand{\wert} {\operatorname{vert}}
\newcommand{\Z}{\mathbb{Z}}
\newcommand{\R}{\mathbb{R}}
\newcommand{\N}{\mathbb{N}}

\begin{document}

\setcounter{page}{1}

\title{Lattice $3$-polytopes with few lattice points}
\author{M\'onica Blanco}
\author{Francisco Santos}

\thanks{Supported by grants MTM2011-22792 (both authors); BES-2012-058920 of the Spanish Ministry of Science and the European Science Foundation within the ACAT Project (M.~Blanco); Alexander von Humboldt Foundation (F. Santos)}

\begin{abstract}
We extend White's classification of empty tetrahedra to the complete classification of lattice $3$-polytopes with five lattice points, showing that, apart from infinitely many of width one, there are exactly nine equivalence classes of them with width two and none of larger width.
We also prove that, for each $n\in \N$, there is only a finite number of (classes of) lattice $3$-polytopes with $n$ lattice points and of width larger than one. This implies that extending the present classification to larger sizes makes sense, which is the topic of subsequent papers of ours.
\end{abstract}

\maketitle

\setcounter{tocdepth}{1}
\tableofcontents


\section{Introduction}
\label{sec:introduction}

A \emph{lattice $d$-polytope} is the convex hull of a finite set of points in $\Z^d$ (or in a $d$-dimensional lattice) containing $d+1$ affinely independent points. We call $\#(P\cap \Z^d)$ the \emph{size} of $P$.
Two lattice polytopes $P$ and $Q$ are said \emph{$\Z$-equivalent} or \emph{unimodularly equivalent} if there is an affine map $t:\R^d\to \R^d$ with $t(\Z^d)=\Z^d$ and $t(P)=Q$. 

Lattice $3$-polytopes of the smallest possible size are \emph{empty tetrahedra}, classified by White some 50 years ago (see Theorem~\ref{thm:empty_tetrahedra}).
Our main result is the next case:

\begin{theorem}
\label{thm:5points}
Every lattice $3$-polytope of size $5$ is $\Z$-equivalent to one listed in Table~\ref{table:5points}. The table is irredundant: polytopes in different rows, or polytopes obtained for different choices of parameters within each row, are not $\Z$-equivalent.

In particular, apart from infinitely many of width one, there are exactly nine (classes of) $3$-polytopes of size $5$ and width two, and none of larger width.
\end{theorem}

Table~\ref{table:5points} includes, apart from the lattice points in a representative for each class, the following invariants of the class (more details about them are in Section~\ref{sec:volumevector}):
\begin{itemize}
\item 
Let $f:\R^d\to \R$ be an affine functional such that $f(\Z^d)\subset \Z$. The integer $\max_{x\in P}f(x) -\min_{x\in P} f(x)$ is called the \emph{width of $P$ with respect to $f$}. The minimum width among all possible (non-constant) choices of $f$ is the \emph{width of $P$}. Hence, $P$ has \emph{width one} if its vertices lie in two consecutive parallel hyperplanes of the lattice.

\item Remember that a set $A$ of $d+2$ points affinely spanning $\R^d$ have a unique (modulo a scalar factor) affine dependence. The \emph{signature} of a $d$-polytope of size $d+2$ is $(i,j)$ if this dependence has $i$ positive  and $j$ negative coefficients. Signatures $(i,j)$ and $(j,i)$ are the same, and the five possible signatures of five points in $\R^3$ are $(4,1)$, $(3,2)$, $(2,2)$, $(3,1)$ and $(2,1)$. 

\item The \emph{volume vector} of a $3$-polytope of size five is a vector in $\Z^5$ recording the volumes of the (perhaps degenerate) tetrahedra spanned by each subset of four of the five points.  All volumes in this paper are ``normalized to the lattice'': the volume of the convex hull of an affine lattice basis equals one, and the volume of every lattice polytope is an integer. 
We give volume vectors a sign convention that makes them have as many positive and negative entries as given by the signature.
\end{itemize}

\begin{table}[h]
\footnotesize 
\begin{tabular}{|c|c|c|c|}
\hline
\textbf{Sign.} &\textbf{Volume vector} & \textbf{ Width} & \textbf{Representative} \\
\hline
$(2,2)$&  $(-1,1,1,-1,0)$ &$1$ & $(0,0,0)$, $(1,0,0)$, $(0,1,0)$, $(1,1,0)$,$(0,0,1)$\\
  \hline
$(2,1)$&  \begin{tabular}{c}
    $(-2q,q,0,q,0)$  \\
    $0\le p\le \frac{q}{2}$,\\ $\gcd(p,q)=1$ 
  \end{tabular}
  & $1$ & $(0,0,0)$, $(1,0,0)$, $(0,0,1)$, $(-1,0,0)$,$(p,q,1)$\\
  \hline
$(3,2)$&  \begin{tabular}{c}
    $(-a-b,a,b,1,-1)$ \\
    $0<a\le b$, \\$\gcd(a,b)=1$
  \end{tabular}
  & $1$  & $(0,0,0)$, $(1,0,0)$, $(0,1,0)$, $(0,0,1)$,$(a,b,1)$\\
  \hline
$(3,1)$&  \begin{tabular}{c}
    $(-3,1,1,1,0)$  \\
    $(-9,3,3,3,0)$
  \end{tabular} &   \begin{tabular}{c}
    $1$ \\ $2$  \end{tabular}  &\begin{tabular}{l}
$(0,0,0)$, $(1,0,0)$, $(0,1,0)$, $(-1,-1,0)$,$(0,0,1)$\\
$(0,0,0)$, $(1,0,0)$, $(0,1,0)$, $(-1,-1,0)$,$(1,2,3)$
  \end{tabular}\\
  \hline
&  $(-4,1,1,1,1)$ & $2$ & $(0,0,0)$, $(1,0,0)$, $(0,0,1)$, $(1,1,1)$,$(-2,-1,-2)$\\
& $(-5,1,1,1,2)$ & $2$ & $(0,0,0)$, $(1,0,0)$, $(0,0,1)$, $(1,2,1)$,$(-1,-1,-1)$\\
& $(-7,1,1,2,3)$ & $2$ & $(0,0,0)$, $(1,0,0)$, $(0,0,1)$, $(1,3,1)$,$(-1,-2,-1)$\\
$(4,1)$&  $(-11,1,3,2,5)$ &$2$ & $(0,0,0)$, $(1,0,0)$, $(0,0,1)$, $(2,5,1)$,$(-1,-2,-1)$\\
&  $(-13,3,4,1,5)$ &  $2$ & $(0,0,0)$, $(1,0,0)$, $(0,0,1)$, $(2,5,1)$,$(-1,-1,-1)$\\
&  $(-17,3,5,2,7)$ &  $2$ & $(0,0,0)$, $(1,0,0)$, $(0,0,1)$, $(2,7,1)$,$(-1,-2,-1)$\\
&  $(-19,5,4,3,7)$ &  $2$ & $(0,0,0)$, $(1,0,0)$, $(0,0,1)$, $(3,7,1)$,$(-2,-3,-1)$\\
&  $(-20,5,5,5,5)$ &  $2$ & $(0,0,0)$, $(1,0,0)$, $(0,0,1)$, $(2,5,1)$,$(-3,-5,-2)$\\
\hline
\end{tabular}
\caption{{Complete classification of lattice $3$-polytopes of size $5$.}}
\label{table:5points}
\end{table}

After some preliminaries on signatures, volume vectors, and empty tetrahedra that we put together in Section~\ref{sec:preliminaries}, we devote Section~\ref{sec:structure} to proving the following structural result for $3$-polytopes of size five.
\begin{theorem}[Theorems~\ref{thm:width1} and~\ref{thm:width2}.]
\label{thm:structure}
Let $P$ be a lattice $3$-polytope of size $5$.
\begin{enumerate}
\item If $P$ has signature $(2,2)$, $(2,1)$ or $(3,2)$, then it has width one.
\item If $P$ has signature $(3,1)$ or $(4,1)$, then there exists an affine integer functional with values $(1,1,0,0,h)$ in the lattice points of $P$, where $h\in\{-1,-2\}$.
\end{enumerate}
\end{theorem}

Once we have this, the proof of Theorem~\ref{thm:5points} goes as follows: 
Polytopes of width one consist of two subconfigurations of sizes $n_1$ and $n_2$ ($n_1+n_2=5$) placed on consecutive parallel  lattice planes. 
The possibilities for the individual subconfigurations are few and easy to find out, so the only complication lies in the possible ``rotations'' (by which we mean elements of $SL(2,\Z)$) of one with respect to the other. We work out the complete and irredundant list of possibilities in Section~\ref{sec:width1}.
For the rest of polytopes of size five, Theorem~\ref{thm:structure} allows for a similar treatment except the subconfigurations (of sizes two, two, and one) now lie in three, instead of two, parallel hyperplanes. This complicates matters, but we still obtain their full classification through a case by case study in Sections~\ref{sec:3-1}, \ref{sec:4-1-nonsymmetric} and \ref{sec:4-1-symmetric}.

\medskip

Let us mention that other approaches to the classification of lattice $3$-polytopes, some overlapping with ours, have been  undertaken:
\begin{itemize}
\item 
Polytopes of signatures $(2,2)$ and $(3,2)$ have width 1 by the following result of Howe~\cite[Thm.~1.3]{Scarf}: Every lattice $3$-polytope with no lattice points other than its vertices has width 1. However, the classification of them included in Table~\ref{table:5points} is, as far as we know, new.

\item 
Polytopes of signature $(4,1)$, which are the same as ``terminal tetrahedra'' or ``clean tetrahedra with a single interior point'', were classified by Kasprzyk~\cite{Kasprzyk} and Reznick~\protect{\cite[Thm.~7]{Reznick-clean}}, who obtained  the same list as ours. In this sense, Sections~\ref{sec:4-1-nonsymmetric} and~\ref{sec:4-1-symmetric} are only reworking the known classifications. Still, we prefer to include them for completeness and because our methods differ from the ones in those papers.

\item After the first version of our paper was available, Averkov et al.~\cite{AverkovKrumpelmannWeltge} proved that the complete list of maximal hollow $3$-polytopes consists only of the $12$ found previously in~\cite{AverkovWagnerWeismantel}. A priori, one could find all $3$-polytopes of size $5$ and of signature different from $(4,1)$ by an exhaustive search among the subpolytopes of these twelve, together with those that project to the second dilation of a unimodular triangle.
\end{itemize}

\medskip

That there is an infinite number of $\Z$-equivalence classes of lattice $3$-polytopes for every size $n \ge 4$ has been previously observed (e.g. in~\cite{LiuZong}). This contrasts with the situation in dimension two, where Pick's Theorem easily implies finitely many classes for each fixed size. Still, both our Theorem~\ref{thm:5points} and White's classical classification of $3$-polytopes of size four (see Theorem~\ref{thm:empty_tetrahedra}) seem to indicate that this infiniteness happens only in width one. In Section~\ref{sec:finiteness} we prove this for every $n$:

\begin{theorem}[Corollary~\ref{coro:finitewidth>1}]
\label{thm:finitewidth>1}
For each $n\ge 4$, there exist finitely many lattice $3$-polytopes of width greater than one and size $n$.
\end{theorem}

This opens the possibility of a complete classification of lattice $3$-polytopes of each fixed size: those of width one admit the same description and classification as the one we use for size five (two subconfigurations in consecutive parallel  lattice planes) and those of width larger than one are a finite list. 

The next case after the one in this paper, the full classification of lattice $3$-polytopes with six points, is undertaken in~\cite{6points}. The techniques used here and in~\cite{6points} could in principle carried over to larger sizes, but the case studies and complications involved would make them unpractical. Instead, in~Section~\ref{sec:finiteness} we sketch a general recursive method to algorithmically  classify  $3$-polytopes of a certain size $n$ and width larger than one if the classification is known for size $n-1$. 
The detailed description and proof of correctness of this method will appear in~\cite{quasiminimal}, together with an implementation giving the full classification up to $n=11$.

\medskip

Let us finish by mentioning that our motivation comes partially from the notion of \emph{distinct pair-sums} lattice polytopes (or dps polytopes, for short), defined as lattice polytopes in which all the pairwise sums $a+b, a,b\in P\cap \Z^d$, are distinct~\cite{ChoiLamReznick}. They are also the lattice polytopes of Minkowski length equal to one, in the sense of~\cite{Beckwith_etal}. For example, $d$-polytopes of size  $d+2$ are dps if and only if their signature is neither $(2,2)$ nor $(2,1)$.
Since dps $d$-polytopes cannot have two lattice points in the same class modulo $(2\Z)^d$, they have size at most $2^d$. 
In particular, \cite{quasiminimal} contains the full classification of the (finitely many) dps $3$-polytopes of width larger than one.

\medskip

\noindent\textbf{Acknowledgment:} We thank Bruce Reznick for pointing us to useful references on the topic of this work.


\section{Preliminaries on lattice \texorpdfstring{$3$}{3}-polytopes}
\label{sec:preliminaries}

We here review some concepts needed in the classification of lattice $3$-polytopes. All the contents are either known or their proofs can be considered routine.

\subsection{Volume vectors}
\label{sec:volumevector}

Since $\Z$-equivalence preserves volume, the following is a useful invariant:

\begin{definition}
\label{def:VolumeVectors}
Let $A=\{  p_1,p_2, \dots ,p_n  \}$, with $n \ge d+1$, be a finite set of lattice points in $\Z^d$.
The \emph{volume vector of $A$} is the vector
\[
{w}= (w_{i_1\dots i_{d+1}})_{1 \le i_1 <\dots <i_{d+1}\le n} \in \mathbb{Z}^{ \binom{n}{d+1}}
\]
where
\begin{equation}
w_{i_1\dots i_{d+1}}
:=\det \left( \begin{array}{ccc}
	1	&\dots	&1\\      
	p_{i_1}	&\dots	&p_{i_{d+1}}\\
\end{array}
\right).
\label{eq:volume}
\end{equation}
\end{definition}

The definition of volume vector implicitly assumes a specific ordering of the $n$ points in $A$. When we say that the volume vector is $\Z$-equivalence invariant, this ordering (and the fact that the sign of each volume entry depends on the ordering) has to be taken into account. 

We now look at the converse question: if two point sets of the same size have the same volume vector, are they necessarily $\Z$-equivalent? The answer is \emph{almost} yes: the volume vector is a~\emph{complete} invariant for $\Z$-equivalence when its gcd equals 1:

\begin{proposition}
\label{prop:VolumeVectors}
Let $A=\{ p_1, \dots , p_n  \}$ and $B= \{ q_1, \dots , q_n \}$ be $d$-dimensional subsets of $\Z^d$ and suppose they have the same volume vector $w=(w_I)_{I \in \binom{[n]}{d+1}}$ with respect to a given ordering. Then:

\begin{enumerate}
\item There is a unique unimodular affine map $t: \R^d\to \R^d$ with $t(A)=B$ (respecting the order of points).

\item If $\gcd_{I \in \binom{[n]}{d+1}} \ (w_I)=1$, then $t$ is a $\Z$-equivalence between $A$ and $B$.

\end{enumerate}
\end{proposition}

\begin{proof}
Without loss of generality we may assume that $w_{1,\dots,d+1} \neq 0$. 
This means that  $\{ p_1, \dots, p_{d+1}  \}$ and $\{ q_1, \dots, q_{d+1} \}$ both span $\R^d$. Then there exists a unique affine map $t: \R^d\to \R^d$ with $t(p_i)=q_i$ for $i\in \{1,\dots,d+1\}$. On the other hand, both $\conv\{ p_1, \dots, p_{d+1}  \}$ and $\conv\{ q_1, \dots, q_{d+1} \}$ have the same volume $w_{1,\dots,d+1}$, which implies that $\det(t)=1$. We claim that $t(p_i)=q_i$ also for $i>d+1$.

To show this, simply observe that for each point $p_i$ with $i>d+1$ the affine dependence on $\{p_1, \dots, p_{d+1},p_i\}$ (which is encoded in the volume vector of $A$) allows to write $p_i$ as an affine combination of $\{p_1,\dots, p_{d+1}\}$. 
Since $t$ preserves affine combinations, $t(p_i)=q_i$. This finishes the proof of part (1).

For part $(2)$, let $\Lambda (A), \Lambda (B) \leq \mathbb{Z}^d$ be the affine sublattices spanned respectively by $A$ and $B$. Since $t$ maps $A$ to $B$,  it maps $\Lambda (A)$ to $\Lambda (B)$. The index $[\mathbb{Z}^d : \Lambda (A)]$ is the minimal volume (with respect to $\mathbb{Z}^d$) of a basis of $\Lambda (A)$. Thus the indices of both $A$ and $B$ divide $w_I$ for all $I \in \binom{[n]}{d+1}$, and therefore they divide $\gcd(w_I)_I$. In particular, if $\gcd(w_I)_I=1$, then  $\Lambda (A) = \mathbb{Z}^d = \Lambda (B) $. This implies $t$ maps $\mathbb{Z}^d$ to itself, so it is a $\Z$-equivalence.
\end{proof}

From the volume vector of a point configuration $A$ with $n \ge d+1$ points, we can recover the volume vector of any subconfiguration. 

Let us look at configurations with $d+2$ points. Remember that if $d+2$ points $\{p_1, \dots ,p_{d+2}\}$  affinely span $\R^d$ then they have a unique (modulo a scalar factor) affine dependence.
The volume vector of $d+2$ points encodes its dependence as follows: let $I_k=\{1, \dots, d+2\} \setminus \{k\}$
\begin{equation}
  \sum_{k=1}^{d+2} (-1)^{k-1} \cdot w_{I_k} \cdot p_{k}=0, \quad \quad
 \sum_{k=1}^{d+2}  (-1)^{k-1} \cdot w_{I_k} =0.
\label{eqn:dependence}
\end{equation}
The points with non-zero coefficient in this dependence form a \emph{circuit}. The \emph{signature} of the circuit is the pair $(i,j)$ if this dependence has $i$ positive and $j$ negative coefficients. 
(See more details in~\cite{deLoeraRambauSantos2010}). We call \emph{signature} of the $d+2$ points the signature of its (unique) circuit; $(i,j)$ and $(j,i)$ are the same signature.

\begin{remark}
For five points $A=\{p_1,\dots,p_5\} \subset \R^d$, the signature $(i,j)$ of $A$ can be $(2,1)$, $(2,2)$, $(3,2)$, $(3,1)$ or $(4,1)$, depicted in Figure~\ref{fig:signatures5points}.

\begin{figure}[ht]
\includegraphics[scale=.9]{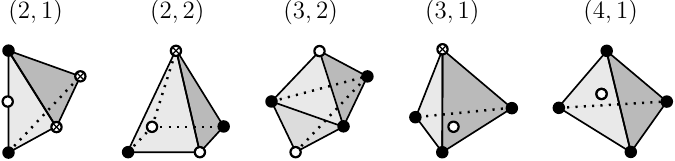}
\caption{The five possible signatures (oriented matroids) of five different points in $\R^3$. Black and white dots represent the positive and negative parts of the circuit, respectively; crossed dots mark points that are not in the circuit.}
\label{fig:signatures5points}
\end{figure}

In this situation we modify our sign and order conventions for writing the volume vector, in order to make the signature (and its correspondence to subsets of $A$) more explicit. More precisely, we take as volume vector for five points $p_1,\dots,p_5$ the vector $(v_1,v_2,v_3,v_4,v_5)$ where 
\[
\sum v_i p_i=0, \qquad \sum v_i=0
\]
is the unique affine dependence on $A$, normalized so that $|v_i|=\vol(\conv(A \setminus \{i\}))$. 
In particular, this way the signature of $A$ equals the number of positive and negative entries in the volume vector.
Put differently (see  Equation~\eqref{eqn:dependence}):
\[
(v_1,v_2,v_3,v_4,v_5) = (w_{2345}, \ -w_{1345},\ w_{1245},\ -w_{1235},\ w_{1234})
\]
where $w_{ijkl}$ is as in Equation~\eqref{eq:volume}. 
\end{remark}

\subsection{Empty tetrahedra}
\label{sec:empty_simplices}
Polytopes of dimension three and of size four---the smallest possible---are
called \emph{empty tetrahedra} since they are lattice tetrahedra without lattice points apart from their vertices.
The analogue of Theorem~\ref{thm:5points} for them is classical. Observe that it implies all empty tetrahedra to have width one:

\begin{theorem}[Classification of empty tetrahedra, White 1964~\cite{White}]
\label{thm:empty_tetrahedra}
Every empty tetrahedron of volume $q$ is unimodularly equivalent to 
\[
T(p,q):= \conv \{ (0,0,0), (1,0,0), (0,0,1), (p,q,1) \},
\]
for some $p\in\Z$ with $\gcd(p,q)=1$. Moreover, $T(p,q)$ is $\Z$-equivalent to $T(p',q)$ if and only if $p'=\pm p^{\pm 1} \pmod q$.
\end{theorem}

We often need to check whether a given tetrahedron $T$ is empty. 
Theorem~\ref{thm:empty_tetrahedra} allows us to proceed as follows: first check that one particular facet of $T$ is empty (equivalently, ``unimodular in the lattice plane containing it''), so that we can unimodularly map this facet to $\conv\{(0,0,0), (1,0,0), (0,1,0)\}$ and the fourth vertex of the tetrahedron to $(a,b,q)$. Then use the following lemma:

\begin{lemma}
\label{lemma:empty(a,b,q)}
The lattice tetrahedron $T=\conv\{(0,0,0),$ $(1,0,0),$ $(0,1,0), (a,b,q)\}$ is empty in $\Z^3$ if, and only if, 
one of the following conditions holds:
\begin{enumerate}
\item[(i)] $a \equiv 1 \pmod q$ and $\gcd(b,q)=1$.
\item[(ii)] $b \equiv 1 \pmod q$ and $\gcd(a,q)=1$.
\item[(iii)] $a + b \equiv 0 \pmod q$ and $\gcd(a,q)=1$.
\end{enumerate}
\end{lemma}

\begin{proof}
By Theorem~\ref{thm:empty_tetrahedra}, $T$ is empty if, and only if, all its edges are primitive and its width equals one with respect to some pair of opposite edges. 
The first equation in parts (i), (ii) and (iii) of the statement expresses width one, respectively, with respect to the three pairs of opposite edges. It is complemented with a condition expressing primitivity of the edges in each pair.
\end{proof}

For future reference we include the following statement which can be read as ``no vertex of an empty tetrahedron is more special than the others''.

\begin{lemma}
\label{lemma:verticesT(p,q)}
Let $u$ be a vertex of the empty tetrahedron $T(p,q)$, for some $1 \le p \le q$, with $\gcd(p,q)=1$. Then, there exists a $\Z$-equivalence sending $u$ to  $(0,0,0)$ and mapping $T(p,q)$ either to itself or to $T(p',q)$, where $p' \equiv p^{-1} \mod q$.
\end{lemma}

\begin{proof}
Recall that the vertices of $T(p,q)$ are $p_0=(0,0,0)$, $p_1=(1,0,0)$, $p_2=(0,0,1)$ and $p_3=(p,q,1)$.
Consider the following transformations $t_i$, $i\in \{1,2,3\}$:
{\small
\[
t_1 (x,y,z)=\left( \begin{array}{ccc}
       -1 	&	0	&	p-1	\\      
	0	&	-1	&	q \\
	0	&	0	&	1
\end{array}\right)\left( \begin{array}{c}
      x\\      
	y\\
	z
\end{array}\right)+\left( \begin{array}{c}
      1\\      
	0\\
	0
\end{array}\right),
\]
\[
t_2 (x,y,z)=\left( \begin{array}{ccc}
       p' 	&	\frac{-pp'+1}{q}	&	0	\\      
	q	&	-q	&	0 \\
	0	&	0	&	-1
\end{array}\right)\left( \begin{array}{c}
      x\\      
	y\\
	z
\end{array}\right)+\left( \begin{array}{c}
      0\\      
	0\\
	1
\end{array}\right),
\]
\[
t_3 (x,y,z)=\left( \begin{array}{ccc}
       -p' 	&	\frac{pp'-1}{q}	&	1-p'	\\      
	-q	&	p	&	-q \\
	0	&	0	&	-1
\end{array}\right)\left( \begin{array}{c}
      x\\      
	y\\
	z
\end{array}\right)+\left( \begin{array}{c}
      p'\\      
	q\\
	1
\end{array}\right).
\]
}
Each $t_i$ sends $p_i$ to $p_0$; $t_1$ maps $T(p,q)$ to itself, while $t_2$ and $t_3$ map it to $T(p',q)$.
\end{proof}

\begin{remark}
\label{rmk:verticesT(p,q)}
The transformation $t_1$ in the proof (exchanging $(0,0,0) \leftrightarrow (1,0,0)$ and $(0,0,1)\leftrightarrow (p,q,1)$) is the only unimodular transformation, other than the identity, sending $T(p,q)$ to itself  for every $p$ and $q$. The other $22$ affine automorphisms of $T(p,q)$ are automorphisms of $\Z^3$ only for particular values of $(p,q)$.

This means that the sentence ``no vertex of an empty tetrahedron is more special than the others'' is not true if we fix a particular class $T(p,q) \subset \Z^3$ of simplices. If we want to stay within a particular class $T(p,q)$, and in this class $p\not\equiv  p^{-1} \pmod q$, then the vertices $(0,0,0)$ and $(1,0,0)$ are in one orbit of the unimodular automorphism group of $T(p,q)$ and $(0,0,1)$ and $(p,q,1)$ in another. 
\end{remark}


\section{A structure theorem for \texorpdfstring{$3$}{3}-polytopes of size five}
\label{sec:structure}

In this section we  prove Theorems~\ref{thm:width1} and~\ref{thm:width2}, the two parts of Theorem~\ref{thm:structure}.

\subsection{A convenient change of coordinates}

When dealing with empty tetrahedra it is often useful to make a change of coordinates so that instead of having a tetrahedron of volume $q$ with respect to $\Z^3$ we have a tetrahedron whose vertices span $\Z^3$ as an affine lattice, but considered as a lattice polytope with respect to a finer lattice. A similar transformation is used, for example, in~\cite{SantosZiegler}.

\begin{proposition}
\label{prop:change-of-coordinates}
Let $p\in\{1,\dots, q\}$  be integers with $\gcd(p,q)=1$.
The linear map $s(x,y,z)= (-y/q+z, y/q, x-py/q)$
maps $T(p,q)$ to the standard tetrahedron 
\[
T_0:=\conv\{o=(0,0,0), e_1=(1,0,0), e_2=(0,1,0), e_3=(0,0,1)\},
\]
and sends $\Z^3$ isomorphically to the lattice
\[
\Lambda(p,q) := \langle(1/q, -1/q, p/q)\rangle + \Z^3.
\]

In particular, $T_0$ is empty in $\Lambda(p,q)$.
\qed
\end{proposition}

Since $T(p,q)$ has width one with respect to the functional $z$ in the integer lattice,  $T_0$ has width one with respect to the functional $x+y$ in the lattice $\Lambda(p,q)$. This implies that
all lattice points of $\Lambda(p,q)$ lie in the family of integer lattice planes  $\{(x,y,z) : x+y\in \Z\}$. This suggests we consider the following rectangle, which is a fundamental rectangle of $\Lambda(p,q) \cap \{x+y =0\}$:
\[
R(p,q):=\conv\{(0,0,0), (0,0,1), (1,-1,0), (1,-1,1)\}.
\]
Also, since the edges $oe_3$ and $e_1e_2$ of $T_0$ are primitive in $\Lambda(p,q)$, the vertices of $R(p,q)$ are the only lattice points in its boundary. Hence \emph{all lattice points of $\Lambda(p,q)\setminus \Z^3$ lie in the relative interior of a unique integer translation of $R(p,q)$}, as illustrated in Figure~\ref{fig:Lambda(p,q)}.

\begin{figure}[htb]
\includegraphics[scale=0.7]{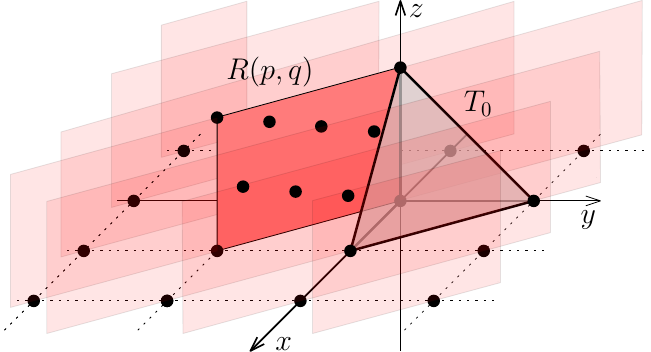}
\caption{The lattice points of $\Lambda(p,q)$ in the standard tetrahedron $T_0$, the plane $\{z=0\}$  and the rectangle $R(p,q)$. Transparent rectangles represent integer translations of $R(p,q)$. (The picture is for $q=7$ and $p=4$).}
\label{fig:Lambda(p,q)}
\end{figure}

In the proof of Theorem~\ref{thm:width1} we need the following result about the 
rectangle $R(p,q)$, illustrated in Figure~\ref{fig:trianglesR(p,q)}.

\begin{lemma}
\label{lemma:trianglesR(p,q)}
Let $q\ge 2$ and let $p\in \{1,\dots,q-1\}$ with $\gcd(p,q)=1$. 
\begin{enumerate}
\item The triangle $\Delta_1:=\conv\{(0,0,0), (1,-1,0), (1,-1,1/2) \}\subset R(p,q)$ contains non-integer points of $\Lambda(p,q)$ if and only if $p\in\{2,\dots,q-2\}$.
\label{item:fund1}
\item The triangle $\Delta_2:=\conv\{(0,0,0), (1,-1,0), (1/2,-1/2,1/2) \}\subset R(p,q)$  contains non-integer points of $\Lambda(p,q)$.
\label{item:fund2}
\end{enumerate}
\end{lemma}

\begin{proof}
Let $p':=p^{-1} \mod q \in \{1,\dots,q-1\}$. The point $(p'/q,-p'/q,1/q)\in \Lambda(p,q)$, and it lies in the interior of $R(p,q)$. 
\begin{enumerate}
\item A point of $R(p,q)$ lies in $\Delta_1$ if and only if it belongs to the halfspace $x-y-4z \ge 0$. Hence $(p'/q,-p'/q,1/q) \in \Delta_1$ if and only if $p' \ge 2$. In the case of $p'=1$ we have that $p=1$ and then every non-integer point of $R(p,q)$ lies in the diagonal $x-y-2z=0$ (and hence not in $\Delta_1$).
\item A point of $R(p,q)$ lies in $\Delta_2$ if and only if it belongs to the halfspaces $x-y-2z \ge 0$ and $x-y+2z -2 \le 0$. Hence $(p'/q,-p'/q,1/q) \in \Delta_1$ for every value of $p'\in\{1,\dots, q-1\}$.
\end{enumerate}
\end{proof}

\begin{figure}[h]
\includegraphics[width=0.8\textwidth]{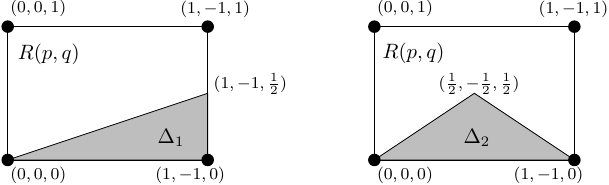}
\caption{Triangles $\Delta_1$ and $\Delta_2$ (gray areas). Black dots represent the integer points of $R(p,q)$.}
\label{fig:trianglesR(p,q)}
\end{figure}


\subsection{Proof of Theorem~\ref{thm:structure}}
\label{sec:proof_of_structure}

\begin{theorem}
\label{thm:width1}
If $P$ is a lattice $3$-polytope of size $5$ and with signature $(3,2)$, $(2,2)$ or $(2,1)$, then $P$ has width one.
\end{theorem}

\begin{proof}
Let $A=\{p_1,p_2,p_3,p_4,p_5\}$ be the lattice points in $P$. We assume the points  ordered so that the volume vector $v=(v_1,v_2,v_3,v_4,v_5)$ of $A$ verifies $v_i\le 0 <v_4\le v_5$, $i=1,2,3$. That is, points $p_4$ and $p_5$ (corresponding to the ``2'' in the signature) lie in opposite sides of the plane generated by the triangle $\conv \{p_1,p_2,p_3\}$.

Consider the empty tetrahedron $T:=\conv \{p_1,p_2,p_3,p_4\}$, of volume $q=v_5$. By Proposition~\ref{prop:change-of-coordinates}, we can consider $T$ to be the standard tetrahedron $T_0$ 
in the lattice $\Lambda(p,q)$ for some $p$ coprime with $q$. Moreover, by Lemma~\ref{lemma:verticesT(p,q)} we can assume $p_4=(0,0,0)$ and, by symmetry of the conditions so far on the points $p_1$, $p_2$ and $p_3$, we can assume that  $p_1=(1,0,0)$, $p_2=(0,1,0)$, $p_3=(0,0,1)$.

The affine dependence $\sum v_i p_i=0$ implies that $p_5=\frac{-1}{v_5}(v_1,v_2,v_3)$ and, since $v_5>0$ and $v_i\le 0$ for $i\in \{1,2,3\}$, $p_5$ lies in the closed positive orthant. Also,  since $\sum v_i=0$ and $v_4\le v_5$, we have $2v_5 \ge v_4+v_5 =-(v_1+v_2+v_3)$. Hence:
\[
p_5\in \{(x,y,z): x+y+z \le 2, x\ge0,y\ge 0, z\ge 0\}= 2T_0.
\]

Remember that all lattice points in $\Lambda(p,q)$ lie in $\{(x,y,z)\in \R^3 : x+y\in \Z\}$. In particular, $P\subset 2T_0$ has width at most two. Moreover, $P$ can have width two only if $p_5$ is one of the points of $\Lambda(p,q)\cap 2T_0$ with $x+y=2$, namely $(2,0,0)$, $(1,1,0)$ and $(0,2,0)$ (see Figure~\ref{fig:signatures-structure}, left). Let us see that in these three cases either $P$ has width one with respect to another functional or $P$ has additional lattice points in the translation $R':=(0,1,0)+ R(p,q)$ of $R(p,q)$, which is a contradiction:
\begin{itemize}
\item If $p_5=(2,0,0)$ then $P\cap R'$ is the triangle $\conv\{(1,0,0),$ $(0,1,0), (1,0,1/2)\}$. By (a translated version of) Lemma~\ref{lemma:trianglesR(p,q)}(\ref{item:fund1}), for this triangle not to contain additional lattice points of $\Lambda(p,q)$ we need $p=1$. But in this case $P$ has width one  with respect to the functional $y+z$.
\item The case $p_5=(0,2,0)$ is analogous, exchanging the roles of $x$ and $y$.
\item If $p_5=(1,1,0)$ then $P\cap R'$  is  $\conv\{(1,0,0),$ $(0,1,0),$ $(1/2,1/2,1/2)\}$. By (a translated version of) Lemma~\ref{lemma:trianglesR(p,q)}(\ref{item:fund2}), this triangle has additional lattice points, independently of the value of $p$, unless $q=1$. But if $q=1$ then $P$ has width one with respect to $z$.
\end{itemize}
\end{proof}

\begin{figure}[h]
\includegraphics[scale=.815]{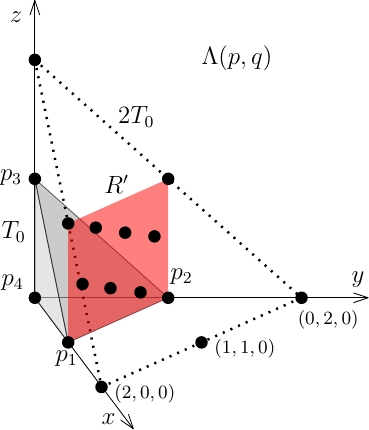}\qquad\includegraphics[scale=.815]{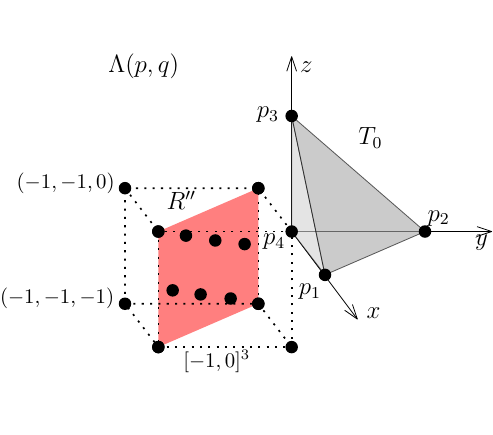}\\
Signature $(2,*)$\qquad\qquad \qquad \qquad Signature $(*,1)$
\caption{The idea in the proofs of Theorems~\ref{thm:width1} and~\ref{thm:width2}. In case of signature  $(2,*)$ (left) the lattice points in $P$ are the four vertices of the standard tetrahedron $T_0$ (in gray) plus a fifth  point guaranteed to lie in $2T_0$ (dotted lines). In case of signature $(*,1)$ (right) the lattice points in $P$ are the four vertices of $T_0$ plus a fifth point  guaranteed to lie in  $[-1,0]^3$ (dotted lines). In both cases, all non-integer lattice points lie in rectangles $R'$ and $R''$, respectively.}
\label{fig:signatures-structure}
\end{figure}

\begin{theorem}
\label{thm:width2}
Let $P$ be a lattice $3$-polytope of size $5$ and signature $(4,1)$ or $(3,1)$. Let $T$ be the empty lattice tetrahedron of largest volume $q\ge1$ contained in $P$. Then there exists an affine integer  functional taking values $1,1,0,0$ in $T$, and $h \in \{-1,-2\}$ in the fifth point.
Moreover, in the case of signature $(4,1)$, having $h=-2$ is equivalent to having a volume vector of the form $(-4q,q,q,q,q)$ (that is, the interior point is the centroid of the other four).
\end{theorem}

\begin{proof}
Let $A=\{p_1,p_2,p_3,p_4,p_5\}$ be the lattice points in $P$ and assume the points are ordered so that the volume vector $v=(v_1,v_2,v_3,v_4,v_5)$ of $A$ verifies $v_4< 0\le v_i\le v_5$, $i=1,2,3$. That is, point $p_4$ (corresponding to the ``1'' in the signature) lies in the tetrahedron $\conv \{p_1,p_2,p_3,p_5\}=P$ and the (empty) tetrahedron $T:=\conv \{p_1,p_2,p_3,p_4\}$, of volume $q=v_5$, has the maximum volume among the empty tetrahedra in $A$.

As in the previous proof, we can assume without loss of generality that $T$ is the standard tetrahedron $T_0$ in the lattice $\Lambda(p,q)$, for some $p$ coprime with $q$, and that its four lattice points are $p_4=(0,0,0)$, $p_1=(1,0,0)$, $p_2=(0,1,0)$ and $p_3=(0,0,1)$. The fifth point is, again, $p_5=\frac{-1}{v_5}(v_1,v_2,v_3)$.  In this case, since $0 \le v_i\le v_5$ for $i\in \{1,2,3\}$, we get that $p_5$ lies in $[-1,0]^3$. More specifically, $p_5$ is either one of the vertices of $[-1,0]^3$, or it is a non-integer point in the translation $R'':=(-1,0,-1)+ R(p,q)$ of $R(p,q)$ (see Figure~\ref{fig:signatures-structure}, right).

This implies the statement for the functional $f(x,y,z)=x+y$. Indeed, $f$ takes only values $\{0,-1,-2\}$ in $[-1,0]^3$, and $f(p_5)=0$ would imply  signature $(2,1)$. 

For the ``moreover'' part, observe that $h=-2$ means $p_5$ to be either $(-1,-1,0)$ or $(-1,-1,-1)$. The first possibility gives signature $(3,1)$, and the second gives signature $(4,1)$ and a symmetric volume vector $(-4q,q,q,q,q)$.
\end{proof}


\section{Classification of $3$-polytopes with five lattice points}
\label{sec:classification}

In Section~\ref{sec:width1} we completely classify $3$-polytopes of size 5 and width one. By Theorem~\ref{thm:width1} this covers signatures $(2,1)$, $(2,2)$ and $(3,2)$. 
In Sections~\ref{sec:3-1},~\ref{sec:4-1-nonsymmetric} and~\ref{sec:4-1-symmetric} we look at signatures $(3,1)$ and $(4,1)$, using the  properties proved in Theorem~\ref{thm:width2}.

\subsection{Polytopes of width \texorpdfstring{$1$}{1}}
\label{sec:width1}


\begin{theorem}
\label{thm:width1_description}
Let $P$ be a lattice polytope of size five and width one. Then $P$ is unimodularly equivalent to one of:
\begin{enumerate}

\item $\conv\{(0,0,0), (1,0,0), (0,1,0), (1,1,0),(0,0,1)\}$, of signature $(2,2)$.

\item $\conv\{(0,0,0), (1,0,0), (0,1,0), (-1,-1,0),(0,0,1)\}$, of signature $(3,1)$.

\item  $\conv\{(0,0,0), (1,0,0), (0,0,1), (-1,0,0),(p,q,1)\}$,  for some $p,q\in \Z$ with $0\le p \le \lfloor q/2\rfloor$ and $\gcd(p,q)=1$. This is of signature $(2,1)$.

\item $\conv\{(0,0,0), (1,0,0), (0,1,0),(0,0,1), (a,b,1)\}$, with $0<a\le b$ and $\gcd(a,b)=1$. This is of signature $(3,2)$.
\end{enumerate}
Moreover, two such polytopes are never $\Z$-equivalent to one another.
\end{theorem}

\begin{proof} 
Width one means the $5$ lattice points of $P$ lie in two consecutive lattice planes. Say $n_0$ points are in $\{z=0\}$ and $5-n_0$ in $\{z=1\}$ with $n_0\ge 5-n_0$. This implies $n_0\in\{3,4\}$.

\begin{itemize}[leftmargin=.7cm]

\item If $n_0=3$, then there are two possibilities:

\begin{itemize}[leftmargin=0.7cm]
\item If the three points at $z=0$ are collinear, without loss of generality we can assume they are $(-1,0,0)$, $(0,0,0)$ and $(1,0,0)$. One of the points at $z=1$ can be assumed to be $(0,0,1)$ and the fifth point has coordinates $(p,q,1)$ with $q\ne 0$ (in order to be full-dimensional) and $\gcd(p,q)=1$ (in order for the edge at $z=1$ to be primitive). The map $(x,y,z)\mapsto (x-y\lfloor p/q+1/2\rfloor,y, z)$ allows us to assume $|p| \le |q|/2$.
Symmetry with respect to the planes $x=0$ and $y=0$ allows us to assume that $0\le p \le \lfloor q/2\rfloor$.

\item If the three points at $z=0$ are not collinear then they form a unimodular triangle, and without loss of generality we assume they are $(0,0,0)$, $(1,0,0)$ and $(0,1,0)$. One of the points at $z=1$ can be assumed to be $(0,0,1)$ and the fifth point has coordinates $(a,b,1)$. By the same argument as before, we need $\gcd (a,b)=1$. By symmetries with respect to the triangle at $z=0$ we can assume $0\le a \le b$ (details are left to the reader).
This configuration has volume vector $(-(a+b),a,b,1,-1)$, so it has signature $(3,2)$ unless $a=0$ (and hence $b=1$ since $\gcd(0,b)=b$). In the case $(a,b)=(0,1)$ we recover the configuration of part (1).
\end{itemize}
\item If $n_0=4$, then the position of the fifth point (within the plane $z=1$) does not affect the $\Z$-equivalence class of $P$, and there are the following three possibilities for the four points at $z=0$. The first two are the configurations of parts (1) and (2). The third one is $\Z$-equivalent to that of part (3) with $(p,q)=(0,1)$:
\medskip

\centerline{
\includegraphics[scale=0.6]{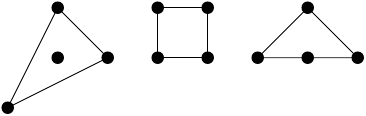}
}

\end{itemize}

This finishes the case study, but we still need to check that different configurations in the list are not $\Z$-equivalent. 
Within those of signature $(3,2)$, since the volume vector is primitive, Proposition~\ref{prop:VolumeVectors} says that different values of $(a,b)$ produce inequivalent configurations. In signature $(2,1)$, however, the volume vector is $ (q,q,0,0,-2q)$ so, a priori, configurations with different $p$ and the same $q$ could still be $\Z$-equivalent. Let us prove that they are not.

For this, let $q$ be fixed and let $p, p' \in \mathbb{Z}$. Let $P$ and $P'$ be two of these configurations having $(p,q,1)$ and $(p',q,1)$ as their fifth point, respectively.
All affine transformations that map $P$ to $P'$ must preserve the collinearity of the three points at $z=0$, so they fix $(0,0,0)$ and either fix or exchange $(1,0,0)$ and $(-1,0,0)$. Similarly, they either fix $(0,0,1)$ and send $(p,q,1)$ to $(p',q,1)$, or they send $(p,q,1)$ to $(0,0,1)$ and $(0,0,1)$ to $(p',q,1)$. So we have four possibilities:
\[
\begin{array}{ll}
(x,y,z)\mapsto\left(x + \frac{p'-p}q y, y , z\right), \quad\ &
(x,y,z)\mapsto\left(x + \frac{-p'-p}q y+p'z, -y +qz, z\right), \\
(x,y,z)\mapsto\left(-x + \frac{p'+p}q y, y , z\right), &
(x,y,z)\mapsto\left(-x + \frac{p-p'}q y+p'z, -y +qz, z\right).
\end{array}
\]
For any of them to be integer we need $p\equiv \pm p'\pmod q$.
\end{proof}

\subsection{Configurations of signature \texorpdfstring{$(3,1)$}{(3,1)}}
\label{sec:3-1}

\begin{theorem}
\label{thm:(3,1)thm}
Every polytope $P$ of signature $(3,1)$ and size $5$ has volume vector equal to $(-3q,q,q,q,0)$ with $q\in\{1,3\}$ and is unimodularly equivalent to one of 
\begin{enumerate}
\item $\conv\{(0,0,0),(1,0,0),(0,0,1),(-1,0,-1), (0,1,0)\}$ (of width one) or
\item $\conv\{(0,0,0),(1,0,0),(0,0,1),(-1,0,-1), (2,3,1)\}$ (of width two).
\end{enumerate}
\end{theorem}

\begin{proof}
For $P$ not to have extra lattice points in the plane containing the $(3,1)$ circuit 
we need the interior point in this coplanarity to be the centroid of the other three. That is, the volume vector must be of the form $(-3q,q,q,q,0)$ (modulo reordering of the points), and all empty subtetrahedra have the same volume.

By Theorem~\ref{thm:width2},  $P$ consists of an empty tetrahedron  $T$ containing two points at each $z=0,1$, and a fifth point at height $h \in \{-1,-2\}$. Without loss of generality, we take the following coordinates:
\[
p_1= (0,0,0), \quad
p_2= (1,0,0), \quad
p_3= (0,0,1), \quad
p_4= (p,q,1), \quad
p_5=(a,b,h).
\]
for some $p \in \Z$ coprime with $q$, and $h=-1,-2$. 

Let us first argue that we can assume $h=-1$.  For this, suppose that $h=-2$ and let us find an affine integer functional $f$ taking values $\{1,1,0,0,-1\}$ in the five points, so that a change of coordinates gives $h=-1$. Of the five lattice points in $P$, both the centroid of the $(3,1)$ circuit and the point that is not in the circuit must lie in the plane $\{z=0\}$ (points $p_1$ and $p_2$). By Lemma~\ref{lemma:verticesT(p,q)} there is no loss of generality in assuming that $p_1$ is the centroid of $p_3$, $p_4$ and $p_5$, so $p_5=3p_1-p_3-p_4=(-p,-q,-2)$.
If $q=1=p$, then the functional $x$ takes the desired values, so assume now that $q>1$.
Lemma~\ref{lemma:empty(a,b,q)} says that in order for the tetrahedron $\conv\{p_1,p_2,p_3,p_5\}$ to be empty we must have one of the following conditions:

\begin{itemize}
\item $p = q-1$ and $\gcd(2,q)=1$. Then take $f(x,y,z)=x-y$.
\item $p = q-2$ and $\gcd(2,q)=1$.  Same, with $f(x,y,z)=x-y+z$.
\item $-2 \equiv 1 \pmod q$ and $\gcd(p,q)=1$. That is, $q=3$ and $p\in \{1,2\}\pmod 3$. This is a particular case of one of the two above, depending on whether $p=2$ or $1$.
\end{itemize}

So for the rest of the proof $h=-1$.  This implies the centroid of the $(3,1)$ circuit is one of $p_1$ or $p_2$ and the point not in the circuit is one of $p_3$ and $p_4$. Then, by Lemma~\ref{lemma:verticesT(p,q)} there is no loss of generality in assuming that $p_1$ is the centroid of the circuit. Also, since the unimodular transformation $f(x,y,z)=(x-pz, -y+qz,z)$ fixes $p_1$ and $p_2$, and sends $p_4 \mapsto p_3$ and $p_4 \mapsto (-p,q,1)$,
there is no loss of generality in assuming that $p_4$ is the point not in the circuit. This implies $p_5=3p_1-p_2-p_3=(-1,0,-1)$. 

The intersection of $P$ with the plane $z=0$ is the triangle with vertices 
\[
p_2= (1,0,0), \quad
\frac{p_3+p_5}{2}=\left(\frac{-1}{2},0,0\right),\quad 
v=\frac{p_4+p_5}{2}=\left(\frac{p-1}{2},\frac{q}{2},0\right).
\]
The condition for $P$ to have size five is that the third vertex $v$ make this triangle not contain any lattice points other than $(0,0,0)$ and $(1,0,0)$. This implies  $q$ to be odd, because if $q$ is even then $v$ itself is a lattice point. (Remember that $\gcd(p,q)=1$). 

\begin{figure}[h]
\includegraphics[scale=1]{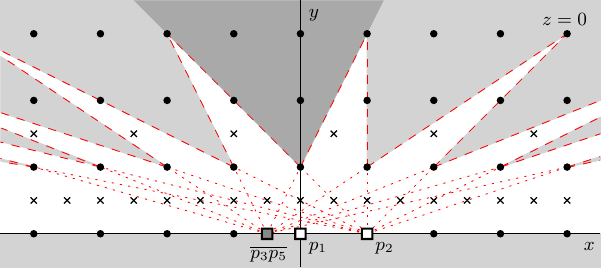}
\caption{{The case analysis in the proof of Theorem~\ref{thm:(3,1)thm}, in the plane $z=0$. White squares represent the points $p_1$ and $p_2$ of $P$. The gray square is the intersection of $p_3p_5$ with the displayed plane. Black dots are the lattice, and black crosses represent the possible intersection points of the edge $p_4p_5$ with the plane $z=0$. 
}}
\label{fig:(3,1)}
\end{figure} 

But other conditions are necessary. For example, in order for $(0,1,0)$ not to be in the triangle, $v$ must be outside the wedge with apex at $(0,1,0)$ and rays in the directions of $(1,2,0)$ and $(-1,1,0)$. (This is the central dark wedge in Figure~\ref{fig:(3,1)}).
The same consideration for the other lattice points of the form $(k,1,0)$ defines analogous wedges so that at the end the only half-integer points not excluded by the wedges are those with $q=1$ or with $q=3$ and $p\equiv2\pmod 3$. If $q=1$ then we get width one, and the configuration is unimodularly equivalent to the first one in the statement. If $q=3$ then all possibilities for $p$ are unimodularly equivalent to one another. Taking $p=2$ we get the second configuration in the statement.
\end{proof}

\subsection{Non-symmetric configurations of signature \texorpdfstring{$(4,1)$}{(4,1)}}
\label{sec:4-1-nonsymmetric}

\begin{theorem}
\label{thm:(4,1)distinct-volumes}
Apart of those with volume vector of the form $(-4q,q,q,q,q)$, every polytope $P$ of size five and signature $(4,1)$ is $\Z$-equivalent to the one whose lattice points are $(0,0,0)$, $(1,0,0)$, $(0,0,1)$ together with one of the following six pairs:
\begin{itemize}
\item $(1,2,1)$ and $(-1,-1,-1)$, volume vector $(-5,1,1,1,2)$.
\item $(1,3,1)$ and $(-1,-2,-1)$, volume vector $(-7,1,1,2,3)$.
\item $(2,5,1)$ and $(-1,-2,-1)$, volume vector $(-11,1,3,2,5)$.
\item $(2,5,1)$ and $(-1,-1,-1)$, volume vector $(-13,3,4,1,5)$.
\item $(2,7,1)$ and $(-1,-2,-1)$, volume vector $(-17,3,5,2,7)$.
\item $(3,7,1)$ and $(-2,-3,-1)$, volume vector $(-19,5,4,3,7)$.
\end{itemize}
\end{theorem}

\begin{proof}
As before, Theorem~\ref{thm:width2} allows us to take the following coordinates:
\[
p_1= (0,0,0), \quad
p_2= (1,0,0), \quad
p_3= (0,0,1), \quad
p_4= (p,q,1), \quad
p_5=(a,-b,-1),
\]
for some $p\in\Z$ with $\gcd(p,q)=1$. (We prefer not to assume $p\in\{1,\dots,q\}$ in this proof, in order to get more symmetric conditions later. The second coordinate in $p_5$ is denoted $-b$ because, as we will soon see, it must be negative).
Without loss of generality (by Lemma~\ref{lemma:verticesT(p,q)}) let $p_1$ be the interior point of $P$. 
Then the volume vector of $P:=\conv{\{p_1,p_2,p_3,p_4,p_5\}}$ is
\[
((a-2)q+bp,-pb-qa,q-b,b,q).
\]

To comply with our hypotheses the five entries must be non-zero, with sign vector $(-,+,+,+,+)$, and the last entry is the biggest among the positive ones (see Theorem~\ref{thm:width2}). This translates into:
\begin{equation}
0<b < q, \qquad 0<-pb-qa\le q.
\label{eqs:4_1}
\end{equation}

We need to find out what values of $a,b,p,q$ make the intersection of $P$ with $\{z=0\}$ not to have other lattice points than $p_1$ and $p_2$. This intersection must  contain $p_1$ in its interior (see Figure~\ref{fig:(4,1)setup}) and it equals the triangle $\Delta$ with vertices 
\[
p_2= (1,0,0), \quad
\frac{p_3+p_5}{2}=\left(\frac{a}{2},\frac{-b}{2},0\right),\quad 
\frac{p_4+p_5}{2}=\left(\frac{p+a}{2},\frac{q-b}{2},0\right).
\]

\begin{figure}[h]
\includegraphics[scale=0.5]{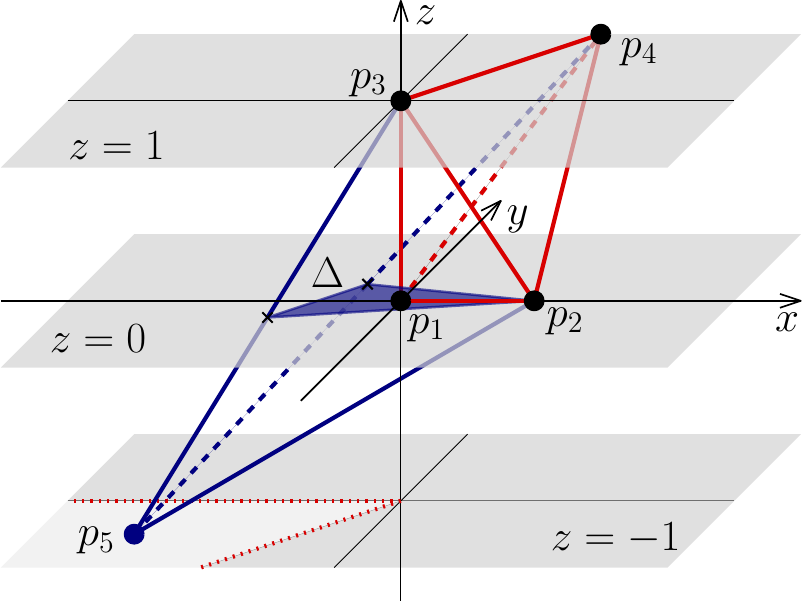}
\caption{The setting for the proof of Theorem~\ref{thm:(4,1)distinct-volumes}.}
\label{fig:(4,1)setup}
\end{figure}

In order to get more symmetric parameters we set $c=p+a$ and $d=q-b>0$, which turns equations~\eqref{eqs:4_1} into
\begin{equation}
b>0,\qquad d>0,
\qquad
0<-ad-bc\le d+b = q.
\label{eqs:4_1b}
\end{equation}
This translates our question into: what values of $a,b,c,d \in \Z$ satisfying equations~\eqref{eqs:4_1b}
have $(0,0)$ and $(1,0)$ as the only lattice points in the triangle 
\[
\Delta=\conv{\{ (1,0), (a/2,-b/2),(c/2,d/2)\}}.
\]

We first make the following two reductions:
\begin{itemize} 
\item \emph{There is no loss of generality in assuming $b\ge d$.} For this, observe that the $\Z$-equivalence $f(x,y,z)=(x-pz, -y+qz,z)$ sends $p_1,\dots, p_5$ to 
\begin{eqnarray*}
p_1=f(p_1)= (0,0,0), \quad
p_2=f(p_2)= (1,0,0), \quad
p'_3:=f(p_4)= (0,0,1), \quad
\\
p'_4:=f(p_3)= (-p,q,1), \quad
p'_5:=f(p_5)=(a+p,-(q-b),-1) = (c,-d,-1),
\end{eqnarray*}
whose parameters $(a',b',c',d')$ are $(c,d,(a+p)-p,q-(q-b))=(c,d,a,b)$. 
\item \emph{There is no loss of generality in assuming $c=1$}.
Via the transformations $(x,y)\mapsto(x\pm y,y)$, we are only interested in $c$ modulo $d$. 
For $d>1$, taking into account that $(c/2,d/2)$ must be outside the wedge symmetric to the triangle $(0,1)p_1p_2$ at point $(0,1)$ we conclude that $c/2 \not\in [1-d/2,0]$, which is equivalent to $c \not\in [2-d,0]$. Thus, the only remaining value for $c\pmod d$ is $c=1$. 
\end{itemize}

\begin{figure}[h]
\centerline{
\includegraphics[scale=0.55]{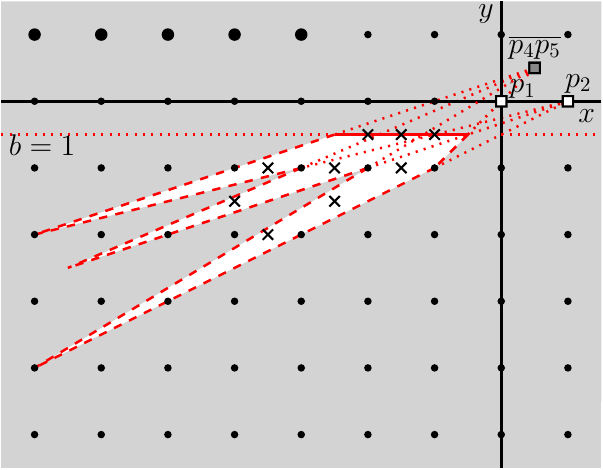}
\includegraphics[scale=0.55]{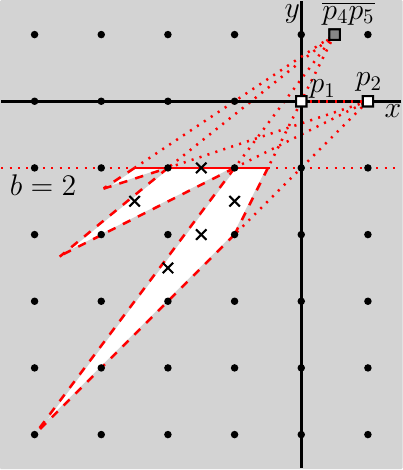}
\includegraphics[scale=0.55]{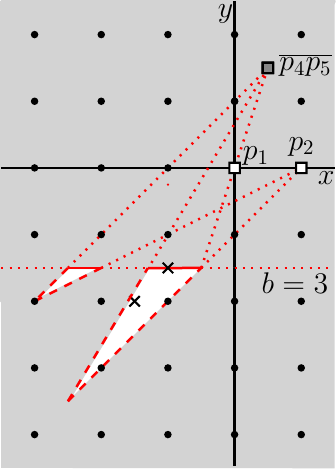}
}
\qquad\qquad $(c,d)=(1,1)$ \hskip 2.5cm $(c,d)=(1,2)$ \hskip 1.3cm $(c,d)=(1,3)$
\caption{The case analysis in the proof of Theorem~\ref{thm:(4,1)distinct-volumes} for  the three possibilities of $(c,d)$. White squares represent the points $p_1$ and $p_2$ of $P$ in the displayed plane $z=0$. The gray square is the intersection of $p_4p_5$ with that same plane. Black dots are the lattice points in the plane and black crosses represent the possible intersection points of the edge $p_3p_5$ and the plane $z=0$. 
}
\label{fig:(4,1)}
\end{figure}

Figure~\ref{fig:(4,1)} shows the possibilities for point $(a,b)$ for the first three cases of $(c,d)$, namely $(c,d)=\{(1,1), (1,2), (1,3)\}$. The figures are read in the same way as Figure~\ref{fig:(3,1)}. Each lattice point in the negative orthant creates an excluded wedge for $(a,b)$. The only novelty is that now we have also an excluded (open) half-plane, the one defined by $d>b$, so that the allowed region (the white region in the pictures) gets smaller and smaller and it becomes lattice-point-free (and eventually empty) for $d\ge 4$ (picture left to the reader).

The $9+5+2$ crosses in the three pictures give a priori  the following $16$ possibilities for the parameters $(a,b,c,d)$:

\medskip

\newcommand{\yes}{\quad \footnotesize{\color{blue} \checkmark}}
\newcommand{\no}{\quad \footnotesize{\color{red} \text{\sffamily X}}}

{\small
\centerline{
\begin{tabular}{|cccc|cc|c|c|}
  \hline
  $a$    & $b$    & $c$    &  $d$  & $p$    &$q$     & \tiny $\gcd(p,q)$ & \tiny $-pb-qa$\\
  \hline
{\bf -2}&{\bf 1} &{\bf 1}&{\bf 1}& {\bf 3}&{\bf 2} &1\yes &1\yes\\
-3        &1        &1        &1        & 4       &2        & 2\no&2\yes\\
-4        &1        &1        &1        & 5       &2        &1\yes& 3\no\\
{\bf -3}&{\bf 2} &{\bf 1}&{\bf 1}& {\bf 4}&{\bf 3}&1\yes&1\yes\\
-5        &2        &1        &1        & 6       &3        &3\no&3\yes\\
-7        &2        &1        &1        & 8       &3        &1\yes&5\no\\
-5        &3        &1        &1        & 6       &4        &2\no&2\yes\\
-8        &3        &1        &1        & 9       &4        &1\yes&5\no\\
\hline
\end{tabular}
\;
\begin{tabular}{|cccc|cc|c|c|}
  \hline  
$a$     &$b$     &$c$    &$d$     & $p$     &$q$    &  \tiny $\gcd(p,q)$ & \tiny $-pb-qa$\\
  \hline    
{\bf -7}&{\bf 4}&{\bf 1}&{\bf 1}& {\bf 8}&{\bf 5}&1\yes&3\yes\\
-3        &2       &1       &2         & 4        &4        &4\no&4\yes\\
{\bf -2}&{\bf 3}&{\bf 1}&{\bf 2} & {\bf 3}&{\bf 5} &1\yes&1\yes\\
-5        &3       &1       &2         & 6        &5        &1\yes&7\no\\
-3        &4       &1       &2         & 4        &6        &2\no&2\yes\\
{\bf -4}&{\bf 5}&{\bf 1}&{\bf 2} & {\bf 5}&{\bf 7}&1\yes&3\yes\\
-2        &3       &1       &3         & 3        &6        &3\no&3\yes\\
{\bf -3}&{\bf 4}&{\bf 1}&{\bf 3} & {\bf 4}&{\bf 7}&1\yes&5\yes\\
\hline
\end{tabular}
}}

\medskip
\normalsize

These $16$ possibilities reduce to only six by excluding those with $\gcd(p,q)\ne1$ (which produce extra lattice points at $z=1$) or $-pb-qa \not\in (0,q]$ (which violate Equation~\eqref{eqs:4_1}). These six are distinguished in boldface in the table above, and give configurations with the following possible pairs for $p_4$ and $p_5$, and their corresponding volume vectors:
\medskip

\centerline{
\begin{tabular}{|c|c|c|}
  \hline
  $p_4=(p,q,1)$ & $p_5=(a,-b,-1)$  & $((a-2)q+bp,\  -pb-qa,\ q-b,\ b,\ q)$\\
  \hline
   $(3,2,1)$ & $(-2,-1,-1)$ & $( -5 , 1 , 1 , 1 , 2 )$\\
  $(4,3,1)$ & $(-3,-2,-1)$ &  $( -7 , 1 , 1 , 2 , 3 )$\\
  $(8,5,1)$ & $(-7,-4,-1)$ &  $( -13 , 3 , 1 , 4 , 5 )$\\
  $(3,5,1)$ & $(-2,-3,-1)$ &  $( -11 , 1 , 2 , 3 , 5 )$\\
  $(5,7,1)$ & $(-4,-5,-1)$ &  $( -17 , 3 , 2 , 5 , 7 )$\\
  $(4,7,1)$ &  $(-3,-4,-1)$ &  $( -19 , 5 , 3 , 4 , 7 )$\\
\hline
\end{tabular}
}
\medskip

Since the volume vectors are all primitive, they completely characterize the configurations (Proposition \ref{prop:VolumeVectors}). The representatives in the statement have been chosen to have smaller coordinates.
\end{proof}

\subsection{Symmetric configurations of signature \texorpdfstring{$(4,1)$}{(4,1)}}
\label{sec:4-1-symmetric}

We finally need to deal with configurations of volume vector $(-4q,q,q,q,q)$.

\begin{theorem}
\label{thm:vvvv-4v}
Every polytope $P$ of size five and signature $(4,1)$ with a symmetric volume vector $(-4q,q,q,q,q)$, is $\Z$-equivalent to the one whose lattice points are $(0,0,0)$, $(1,0,0)$, $(0,0,1)$ together with one of the following two pairs:
\begin{itemize}
\item $(1,1,1)$ and $(-2,-1,-2)$, volume vector $(-4,1,1,1,1)$.
\item $(2,5,1)$ and $(-3,-5,-2)$, volume vector $(-20,5,5,5,5)$.
\end{itemize}
\end{theorem}
Notice that both configurations have width two, with respect to $f(x,y,z)=x-z$.

\begin{proof}
Theorem~\ref{thm:width2}, using that $\{p_1,\dots,p_4\}$ form an empty tetrahedron and that the volume vector is $(-4q, q,q,q,q)$, allows us to take the following coordinates for the lattice points of $P$:
\[
p_1= (0,0,0), \ 
p_2= (1,0,0), \ 
p_3= (0,0,1), \ 
p_4= (p,q,1), \ 
p_5=(-p-1,-q,-2),
\]
for some $1 \le p \le q$ with $\gcd(p,q)=1$.

The convex hull of $P$ consists of four thetrahedra glued together, all of normalized volume $q$; the one we started with and the following three:
\begin{itemize}
\item $T_1=\{(0,0,0),(1,0,0),(0,0,1), (-p-1,-q,-2)\}$,
\item $T_2=\{(0,0,0),(1,0,0),(p,q,1), (-p-1,-q,-2)\}$, and
\item $T_3=\{(0,0,0),(0,0,1),(p,q,1), (-p-1,-q,-2)\}$.
\end{itemize}
We need to check what values of $p$ and $q$ make these three tetrahedra empty. If $q=1=p$, then all tetrahedra are unimodular and therefore empty. This corresponds to the first configuration in the statement. Assume $q>1$ for the rest of the proof.
 
Lemma~\ref{lemma:empty(a,b,q)} says that in order for the tetrahedron $T_1=\conv\{p_1,p_2,p_3,p_5\}$  to be empty we need one of the following conditions:

\begin{itemize}
\item $p = q-2$ and $\gcd(2,q)=1$. 
\item $-2 \equiv 1 \pmod q$ and $\gcd(p+1,q)=1$. That is, $q=3$ and $p=1$. This is a particular case of the one above.
\item $p = q-3$ and $\gcd(q-2,q)=1$. 
\end{itemize}
That is, $q$ is odd and $p\in\{q-2,q-3\}$.

In order for the tetrahedron $T_2=\conv\{p_1,p_2,p_4,p_5\}$ to be empty, the same Lemma~\ref{lemma:empty(a,b,q)} says that we need one of the following conditions:

\begin{itemize}
\item $p = 2$ and $\gcd(2,q)=1$. 
\item $-2 \equiv 1 \pmod q$ and $\gcd(p-1,q)=1$. That is, $q=3$ and $p=2$. This is a particular case of the one above.
\item $p = 3$ and $\gcd(q-2,q)=1$.  
\end{itemize}
That is, $q$ is odd and $p=2$ or $p=3$.
\medskip

This implies $q=5$ and $p \in \{2,3\}$, which makes $T_3$ have width 1 as well: 
$\{(0,0,0),$ $(0,0,1),$ $(2,5,1), (-3,-5,-2)\}$ has width one with respect to $y-2x$, and
$\{(0,0,0),(0,0,1),$ $(3,5,1),$ $(-4,-5,-2)\}$ has width one with respect to $y+z-2x$.

A priori, this could lead to two different configurations with $q=5$. Not surprisingly, the following matrix represents a $\Z$-equivalence mapping the configuration with $p=2$ to the one with $p=3$:
\[
\left( \begin{array}{ccc}
	1	&-1	&3\\      
	0	&-1	&5\\
	0	&0	&1
\end{array}
\right)
\]
\end{proof}


\section{Towards a classification of all lattice $3$-polytopes}
\label{sec:finiteness}

For every $d\ge 3$ and for every $n\ge d+1$ it is easy to construct infinitely many classes of lattice $d$-polytopes of size $n$ \cite[Theorem 4]{LiuZong}. It is known, however, that this cannot happen if we look only at polytopes with interior lattice points:

\begin{theorem*}[Hensley~\protect{\cite[Thm.~3.6]{Hensley}}]
For each $k\ge 1$ there is a number $V(k,d)$ such that no lattice $d$-polytope with $k$ interior lattice points has volume above $V(k,d)$.
\end{theorem*}

\begin{theorem*}[Lagarias-Ziegler~\protect{\cite[Thm.~2]{LagariasZiegler}}]
For each $V\in \N$ there is only a finite number of $\Z$-equivalence classes of $d$-polytopes with volume $V$ or less.
\end{theorem*}

Lattice polytopes with no lattice points in their interior are called \emph{hollow}. For hollow polytopes, although infinitely many for each size, we still have:

\begin{theorem*}[Nill-Ziegler~\protect{\cite[Thm.~1.2]{NillZiegler}}]
There is only a finite number of hollow  $d$-polytopes that do not admit a lattice projection onto a hollow  $(d - 1)$-polytope.
\end{theorem*}


Combining these three statements with the fact that there is a unique hollow $2$-polytope of width larger than one, we get:

\begin{corollary}
\label{coro:finitewidth>1}
For each $n\ge4$, there are finitely many lattice $3$-polytopes of width greater than one and size $n$.
\end{corollary}

\begin{proof}
Once we fix $n$, every lattice $3$-polytope $P$ with $n$ lattice points falls in one of the following (not mutually exclusive) categories:
\begin{enumerate}
\item It is not hollow. In this case Hensley's Theorem gives a bound for its volume. This, in turn, implies finiteness via the Lagarias-Ziegler Theorem.
\item It is hollow, but does not project to a hollow $2$-polytope. These are a finite family, by the Nill-Ziegler Theorem.
\item It is hollow, and it projects to a (hollow) $2$-polytope of width $1$. This implies that $P$ itself has width $1$.
\item It is hollow and it projects to a hollow $2$-polytope of width larger than one. The only such $2$-polytope is the second dilation of a unimodular triangle. 
It is easy to check that only finitely many ($\Z$-equivalence classes of) $3$-polytopes of size $n$ project to it: let $P=\conv \{ p_1,...,p_n\}$ be a $3$-polytope of size $n$ that projects onto $T=\conv \{ (0,0),(2,0),(0,2)\}$.

We must have at least one point projecting to each vertex of $T$. That is: there are $p_1=(0,0,z_1)$, $p_2=(2,0,z_2)$ and $p_3=(0,2,z_3)$ in $P$. The unimodular transformation 
\[
(x,y,z)\mapsto \left(x,y, z-z_1-x \left\lfloor\frac{z_2-z_1}{2}\right\rfloor-y \left\lfloor\frac{z_3-z_1}{2}\right\rfloor\right)
\]
allows us to assume that $z_1,z_2,z_3\in \{0,1\}$. This implies  that $P\subset T\times [1-n,n]$, so there is a finite number of possibilities for $P$.
\end{enumerate}\end{proof}

\begin{remark}
\label{rmk:upperbound}
One may ask how does the number of $3$-polytopes of width $\ge2$ grows with $n$. The bottleneck to this is the huge bound for the polytopes in case (1) of the previous proof (those of type (2) are a constant number, independent of $n$, and those of type (4) grow polynomially). For those of type (1), the bound is (asymptotically) the same as the number of $3$-polytopes with $n$ interior lattice points.
\end{remark}

\begin{remark}
\label{rmk:higherdimensions}
The following higher dimensional analogue of Corollary~\ref{coro:finitewidth>1} is proven in~\cite{threshold}: for each dimension $d$ there is a constant $w(d)\in \N$ such that for every $n$ there are finitely many $d$-polytopes of size $n$ and width greater than $w(d)$. Corollary~\ref{coro:finitewidth>1} says that $w(3)=1$ and the main result in~\cite{threshold} is that $w(4)=2$.
\end{remark}

Still, it is not clear how to find all the (finitely many) $3$-polytopes of width larger than one for each given size $n$. We sketch here the method that we implement in~\cite{quasiminimal}.

Let $P \subset \R^d$ be a lattice $d$-polytope, which we assume to have width greater than one. 
For each vertex $v$ of $P$, we denote $P^v:=\conv (P\cap \Z^d \setminus \{v\})$. 
Let $\wert(P)$ be the set of all vertices and $ \wert^*(P) \subseteq  \wert(P)$ be the set of vertices of $P$ such that $P^v$ is either $(d-1)$-dimensional or has width one.

\begin{definition}
\label{def:(quasi)minimal}
Let $P$ a lattice $d$-polytope $P$ of width $>1$. We say that $P$ is \emph{minimal} if $\wert^* (P)=\wert (P)$ and
\emph{quasi-minimal} if   $\#\wert^*(P) \ge \#  \wert(P)-1$.
\end{definition}

For example, since all $3$-polytopes of size $4$ have width one,  all $3$-polytopes of size $5$ and width $>1$ are minimal.

Our interest in quasi-minimal $3$-polytopes comes from the following observation: 
If we can classify all quasi-minimal $3$-polytopes of a certain size $n$, then we can easily construct the rest of lattice $3$-polytopes of width $>1$ as the (convex hull of the) union of two smaller polytopes of width larger than one. 

Let us be more specific: if a polytope $P$ of width larger than one is not quasi-minimal, then it contains two proper lattice subpolytopes $P^{v_1}$ and $P^{v_2}$ of width larger than one (and size $n-1$)  with
\[
P\cap \Z^3 = (P^{v_1}\cap \Z^3)\cup (P^{v_2}\cap \Z^3).
\]
Then, one of two things happens:
\begin{itemize}
\item If $P^{v_1}\cap P^{v_2}\cap \Z^3$ is still $3$-dimensional, we can think of $P$ as being obtained by gluing $P^{v_1}$ and $P^{v_2}$, and there are finitely many possible ways of gluing two given polytopes in this fashion: fixing an affine basis in $P^{v_1}$ and its (ordered) image in $P^{v_2}$ fixes the gluing.
That is, we can enumerate polytopes $P$ of this type by gluing smaller polytopes, which we assume recursively classified.

\item If $Q=\conv(P_1\cap P_2\cap \Z^3)$ is not three-dimensional, then the number of ways of gluing is more difficult to control, but the possibilities for $Q$ are easy to study. In fact, we show in~\cite{quasiminimal} that this situation implies $Q$ to have size at most four and hence $P$ to have size at most six.
\end{itemize}

Thus, the classification of all lattice $3$-polytopes reduces to that of quasi-minimal $3$-polytopes, plus an implementation of the gluing algorithm.
One problem with this approach is that there are infinitely many quasi-minimal $3$-polytopes:
\begin{proposition}
\label{prop:inf_minimal}
There exist infinitely many minimal (and hence quasi-minimal) $3$-polytopes.
\end{proposition}

\begin{proof}
For every $k\ge 2$, the polytope 
\[
P_k:=\conv\{(1,0,0),(-1,0,0),(0,-1,k),(0,1,k)\}
\]
 has $k-1$ interior points, volume $4k$, and is minimal: 
$P_k^{(1,0,0)}$ and $P_k^{(-1,0,0)}$ have width $1$ with respect to $x$; 
$P_k^{(0,1,k)}$ and $P_k^{(0,-1,k)}$ have width $1$ with respect to $y$.
%
\end{proof}

In dimension one the only quasi-minimal polytope is a segment of length two (which is in fact minimal).
In dimension two, there are the following quasi-minimal polygons (see \protect{\cite[Lemma~2.8]{quasiminimal}}). The four in the top row are minimal, and in the others a white dot indicates the vertex not in $\wert^*(P)$.

\medskip
\centerline{
\includegraphics[scale=0.55]{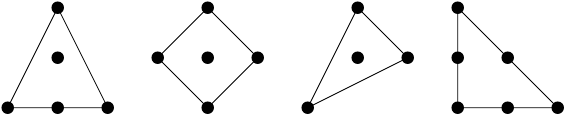}
}
\bigskip
\centerline{
\includegraphics[scale=0.55]{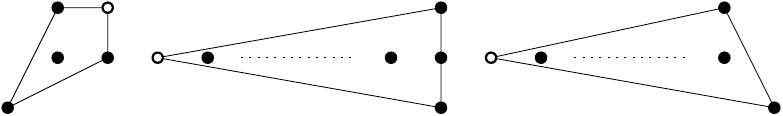}
}

In dimension $3$, the infinite families of quasi-minimal $3$-polytopes can easily be understood, thanks to the following structure theorem. Its proof, together with the complete classification of quasi-minimal $3$-polytopes, can be found in~\cite{quasiminimal}.

\begin{theorem}[\protect{\cite[Theorem~1.3]{quasiminimal}}]
\label{thm:spiked_intro}
Let $P$ be a quasi-minimal lattice $3$-polytope with more than $11$ lattice points. Then $P$ projects to one of the following $2$-polytopes in such a way that all of the vertices in the projection have a unique element in the pre-image. 
\medskip

\centerline{
\includegraphics[scale=0.55]{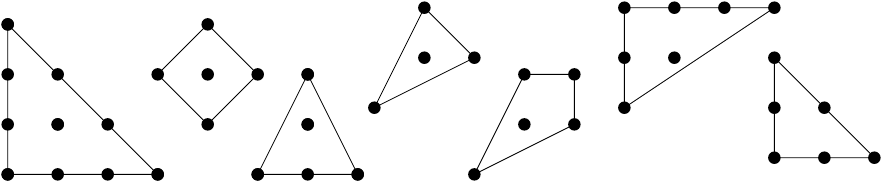}
}
\end{theorem}


\begin{thebibliography}{99}
\bibitem{AverkovWagnerWeismantel}
G.~Averkov, C.~Wagner and R.~Weismantel.
\newblock Maximal lattice-free polyhedra: finiteness and an explicit description in dimension three.
{\tt arXiv:1010.1077}

\bibitem{AverkovKrumpelmannWeltge}
G.~Averkov, J.~Kr\"umplemann and S.~Weltge.
\newblock Notions of maximality for integral lattice-free polyhedra: the case of dimension three.
{\tt arXiv:1509.05200}

\bibitem{Beckwith_etal}
Olivia Beckwith, Matthew Grimm, Jenya Soprunova, Bradley Weaver.
\newblock Minkowski length of 3D lattice polytopes.
\newblock {\em Discrete Comput. Geom.} 48 (2012), 1137--1158.

\bibitem{threshold}
M.~Blanco, C.~Haase, J.~Hoffman, F.~Santos, 
The finiteness threshold width of lattice polytopes,
in preparation.

\bibitem{6points}
M.~Blanco and F.~Santos.
\newblock Lattice 3-polytopes with six lattice points.
\newblock {\em SIAM J. Discrete Math.} 30 (2016), no.~2, 687--717.

\bibitem{quasiminimal}
M.~Blanco and F.~Santos.
\newblock On the enumeration of lattice $3$-polytopes.
\newblock Preprint, January 2016, {\tt arXiv:1601.02577}.

\bibitem{ChoiLamReznick}
M.~D.~Choi, T.~Y.~Lam and B.~Reznick.
\newblock Lattice polytopes with distinct pair-sums.
\newblock {\em Discrete Comput. Geom.} 27 (2002), 65--72.


\bibitem{deLoeraRambauSantos2010}
J.~A.~De Loera,~J. Rambau,~F. Santos.
\newblock
\emph{Triangulations: Structures for Algorithms and Applications.} Algorithms and Computation in Mathematics, Vol.~25. Springer-Verlag, 2010.

\bibitem{Hensley}
D.~Hensley.
\newblock Lattice vertex polytopes with interior lattice points.
\newblock {\em Pacific J. Math.} 105:1 (1983), 183--191.

\bibitem{Kasprzyk}
A.~Kasprzyk. 
Toric Fano 3-folds with terminal singularities. 
\emph{Tohoku Math. J.} (2) 58:1 (2006), 101--121.

\bibitem{LagariasZiegler}
J.~C.~Lagarias and G.~M.~Ziegler.
\newblock Bounds for lattice polytopes containing a fixed number of interior points in a sublattice.
\newblock {\em Canadian J.~Math.}, 43:5 (1991), 1022--1035.

\bibitem{LiuZong}H.~Liu and C.~Zong,
On the classification of convex lattice polytopes,
\emph{Advances in Geometry}, 14(2) (2014), 239--252. 
DOI: 10.1515/advgeom-2013-0022

\bibitem{NillZiegler}
B.~Nill and G.~M.~Ziegler.
\newblock Projecting lattice polytopes without interior lattice points.
\newblock {\em Math. Oper. Res.}, 36:3 (2011), 462--467.

\bibitem{Reznick-clean}
B.~Reznick.
\newblock Clean lattice tetrahedra, 
preprint, June 2006, 21~pages, {\tt arXiv:math/0606227}.


\bibitem{SantosZiegler}
F.~Santos, G.~M. Ziegler.
\newblock Unimodular triangulations of dilated 3-polytopes.
\newblock \emph{Trans.~Moscow Math.~Soc.}, 74 (2013), 293--311. DOI: 10.1090/S0077-1554-2014-00220-X 


\bibitem{Scarf}
H.~E.~Scarf.
\newblock Integral polyhedra in three space.
\newblock {\em Math. Oper. Res.} 10:3 (1985), 403--438.

\bibitem{White}
G.~K.~White.
\newblock Lattice tetrahedra.
\newblock {\em Canadian J.~Math.} 16 (1964), 389--396.


\end{thebibliography}
\end{document}